\documentclass[letterpaper, 10 pt, conference]{ieeeconf}  

\IEEEoverridecommandlockouts                              

\usepackage{graphics} 
\usepackage{amsmath} 
\usepackage{amssymb}  

\usepackage{dsfont}
\usepackage{float}

\usepackage{tikz}
\usetikzlibrary{patterns}

\newcommand{\mb}[1]{\mathbb{#1}}
\newcommand{\mc}[1]{\mathcal{#1}}

\newcommand{\supp}{\operatorname{supp}} 

\usepackage{ntheorem}

\newtheorem{theorem}{Theorem}
\newtheorem{lemma}{{Lemma}}
\newtheorem{prop}{{Proposition}}
\newtheorem{propi}{{Property}}

\theorembodyfont{\upshape}
\newtheorem{definition}{Definition}
\newtheorem{cond}{Condition}
\theoremheaderfont{\upshape}

\newtheorem{rmq}{\textit{Remark}}

\title{\LARGE \bf
Controllability and optimal control of the transport equation with a localized vector field*
}

\author{Michel Duprez$^{1}$ and Morgan Morancey$^{2}$ and Francesco Rossi$^{3}$
\thanks{*This work was supported by Archim\`ede Labex (ANR-11-LABX-0033) and of the A*MIDEX project (ANR-11-IDEX-0001-02), funded by the ``Investissements d'Avenir" French Government programme managed by the French National Research Agency (ANR). The second and third authors acknowledge the support of the ANR project CroCo ANR-16-CE33-0008.}
\thanks{$^{1}$Aix Marseille Universit\'{e}, CNRS, Centrale Marseille, I2M, LSIS,  Marseille, France.
        {\tt\small mduprez@math.cnrs.fr}}%
\thanks{$^{2}$Aix Marseille Universit\'{e}, CNRS, Centrale Marseille, I2M, Marseille, France.
        {\tt\small morgan.morancey@univ-amu.fr}
        }%
\thanks{$^{3}$Aix Marseille Universit\'{e}, CNRS, ENSAM, Universit\'{e}\ de Toulon, LSIS, Marseille, France.	 
        {\tt\small francesco.rossi@lsis.org}
        }%
}

\begin{document}

\maketitle
\thispagestyle{empty}
\pagestyle{empty}

\begin{abstract}
We study controllability of a Partial Differential Equation of transport type, that arises in crowd models. We are interested in controlling such system with a control being a Lipschitz vector field on a fixed control set $\omega$.

We prove that, for each initial and final configuration, one can steer one to another with such class of controls only if the uncontrolled dynamics allows to cross the control set $\omega$.

We also prove a minimal time result for such systems. We show that the minimal time to steer one initial configuration to another is related to the condition of having enough mass in $\omega$ to feed the desired final configuration.
\end{abstract}

\section{INTRODUCTION}

\quad In recent years, the study of systems describing a crowd of interacting autonomous agents 
has draw a great interest from the control community (see e.g. the Cucker-Smale model \cite{CS}).
A better understanding of such interaction phenomena can have a strong impact in several key applications, such as road traffic and egress problems for pedestrians. Beside the description of interaction, it is now relevant to study problems of {\bf control of crowds}, i.e. of controlling such systems by acting on few agents, or on the crowd localized in a small subset of the configuration space.

Two main classes are widely used to model crowds of interacting agents. In {\bf microscopic models}, the position of each agent is clearly identified; the crowd dynamics is described by a large dimensional ordinary differential equation, in which couplings of terms represent interactions. In {\bf macroscopic models}, instead, the idea is to represent the crowd by the spatial density of agents; in this setting, the evolution of the density solves a partial differential equation of transport type. This is an example of a {\bf distributed parameter system}. Some nonlocal terms can model the interactions between the agents. In this article, we focus on this second approach.

To our knowledge, there exist few studies of control of this kind of equations.
In \cite{PRT15}, the authors provide approximate alignment of a crowd described by the Cucker-Smale model \cite{CS}. The control is the acceleration, and it is localized in a control region $\omega$ which moves in time. In a similar situation, a stabilization strategy has been established in  \cite{CPRT17}, by generalizing the Jurdjevic-Quinn method to distributed parameter systems.

%
%

In this article, we study a partial differential equation of transport type, that is widely used for modeling of crowds. Let $\omega$ be a nonempty open connected subset 
of $\mb{R}^d$ ($d\geq1$), being the portion of the space on which the control is allowed to act.
Let $v: \mb{R}^d\rightarrow\mb{R}^d$ be a vector field assumed Lipschitz and uniformly bounded.
Consider the following  linear transport equation 
\vspace*{-1mm}\begin{equation}\label{eq:transport}
	\left\{
	\begin{array}{ll}
\partial_t\mu +\nabla\cdot((v+\mathds{1}_{\omega}u)\mu)=0&\mbox{ in }\mb{R}^d\times\mb{R}^+,\\\noalign{\smallskip}
\mu(\cdot,0)=\mu^0&\mbox{ in }\mb{R}^d,\\
	\end{array}
	\right.\vspace*{-1mm}
\end{equation}
where $\mu(t)$ is the time-evolving measure representing the crowd density and $\mu^0$ is the initial data. 
The control is the function $\mathds{1}_{\omega}u:\mb{R}^d\times\mb{R}^+\rightarrow\mb{R}^d$.
The function $v+\mathds{1}_{\omega}u$ represents the velocity field acting on $\mu$.
System \eqref{eq:transport} is a first approximation for crowd modeling, since the uncontrolled vector field $v$ is given, and it does not describe interactions between agents. Nevertheless, it is necessary to understand controllability properties for such simple equation. Indeed, the results contained in this article will be instrumental to a forthcoming paper, where we will study more complex crowd models, with a non-local term $v[\mu]$.

We now recall the precise notion of approximate controllability  for System \eqref{eq:transport}. 
We say that
System \eqref{eq:transport} is 
{\it approximately controllable}
from $\mu^0$ to $\mu^1$ on the time interval $(0,T)$ if for each $\varepsilon>0$ there exists $\mathds{1}_{\omega}u$ 
such that the corresponding 
solutions to System \eqref{eq:transport} satisfies $W_p(\mu(T),\mu^1)\leqslant \varepsilon$.
%
The definition of the Wasserstein distance $W_p$ is recalled in Section \ref{section 2}.

To control System \eqref{eq:transport},  from a geometrical point of view, the uncontrolled vector field $v$ needs to send the support of $\mu^0$ to $\omega$ forward in time and the support of $\mu^1$ to $\omega$ backward in time. 
This idea is formulated in the following Condition:

\begin{cond}[Geometrical condition]\label{cond1}
Let $\mu^0,\mu^1$ be two probability measures on $\mb{R}^d$ satisfying:
\begin{enumerate}
\item[(i)] For all $x^0\in \supp(\mu^0)$, 
there exists $t^0>0$ such that $\Phi_{t^0}^v(x^0)\in \omega,$
where $\Phi_{t}^v$ is the \textit{flow} associated to $v$, \textit{i.e.} the solution to the Cauchy problem 
\vspace*{-1mm}\begin{equation*}
\left\{\begin{array}{l}
\dot x(t) =v(x(t))\mbox{ for a.e. }t>0,\\\noalign{\smallskip}
x(0)=x^0.
\end{array}\right.\vspace*{-1mm}
\end{equation*}

\item[(ii)] For all $x^1\in \supp(\mu^1)$, 
there exists $t^1>0$ such that $\Phi_{-t^1}^{v}(x^1)\in \omega$.
\end{enumerate}
\end{cond}

\begin{rmq}
Condition \ref{cond1} is the minimal one that we can expect to steer any initial condition to 
any targets.
Indeed, if the first item of Condition \ref{cond1} is not satisfied, there exists a whole sub-population of the measure $\mu_0$ that never intersects the control region, thus, we cannot act on it.
\end{rmq}

%

We denote by $\mc{U}$ the set of admissible controls, that are functions $\mathds{1}_{\omega}u:\mb{R}^d\times\mb{R}^+\rightarrow\mb{R}^d$ Lipschitz in space, 
measurable in time and uniformly bounded. 
 If we impose the classical Carath\'eodory condition of $\mathds{1}_{\omega}u$ being in $\mc{U}$, then the flow 
$\Phi^{v+\mathds{1}_{\omega}u}_t$ is an homeomorphism (see \cite[Th. 2.1.1]{BP07}).
As a result, one cannot expect  exact 
controllability, since for general measures there exists no homeomorphism sending one to another.
We then have the following result of approximate controllability.

\begin{theorem}\label{th1}
Let  $\mu^0,\mu^1$ be two probability measures on $\mb{R}^d$ 
compactly supported absolutely continuous with respect to the Lebesgue measure and satisfying Condition \ref{cond1}. 
Then there exists $T>0$ 
such that System \eqref{eq:transport} is approximately controllable at time $T$
from $\mu^0$ to $\mu^1$
with a control $\mathds{1}_{\omega}u$ in $\mc{U}$.
\end{theorem}

The proof of this result will be given in Section \ref{sec:lip}.
After having proven approximate controllability for System \eqref{eq:transport}, we aim to study the minimal time problem, i.e. the minimal time to send $\mu_0$ to $\mu_1$. We have the following result.

\begin{theorem}\label{th2}
Let  $\mu^0,~\mu^1$ be two probability measures, with compact support, absolutely continuous with respect to the Lebesgue measure and satisfying Condition \ref{cond1}. 

We say that $T^*$ is an admissible time if it satisfies 
\begin{enumerate}
\item[(a)] For each $x^0\in\supp(\mu^0)$
\vspace*{-2mm}$$T^*\geqslant\inf\{t\in\mb{R}^+:\Phi_t^v(x^0)\in \omega\}.\vspace*{-2mm}$$
\item[(b)] For each $x^1\in\supp(\mu^1)$
\vspace*{-2mm}$$T^*\geqslant\inf\{t\in\mb{R}^+:\Phi_{-t}^v(x^1)\in \omega\}.\vspace*{-2mm}$$
\item[(c)]There exists a sequence $(u_k)_k$ of $\mc{C}^{\infty}$-functions equal to $0$ in $\omega^c$
such that
\vspace*{-1mm}\begin{equation}
\lim\limits_{k\rightarrow\infty}[\Phi_t^{v+u_k}\#\mu^0](\omega)
\geqslant 1-\lim\limits_{k\rightarrow\infty}[\Phi_{t-T^*}^{v+u_k}\#\mu^1](\omega).\label{cond temps opt}
\vspace*{-1mm}\end{equation}
\end{enumerate}

Let $T_0$ be the infimum of such $T^*$. Then, for all $T>T_0$, System \eqref{eq:transport} is approximately controllable from $\mu^0$ to $\mu^1$ at time $T$.
\end{theorem}

The proof of this Theorem is given in Section \ref{s-proof2}.

\begin{rmq}
The meaning of condition \eqref{cond temps opt} is the following: functions $u_k$ are used to store the mass in $\omega$. Thus, condition \eqref{cond temps opt} means that at each time $t$ there is more mass that has entered $\omega$ that mass that has exited. This is the minimal condition that we can expect in this setting, since control can only move masses, without creating them.
\end{rmq}

This paper is organized as follows. 
In Section \ref{section 2}, we recall some properties of the continuity equation and the Wasserstein distance.
Sections \ref{sec:lip} and \ref{s-proof2} are devoted to prove Theorems \ref{th1} and \ref{th2}, respectively.
We conclude with some numerical examples in Section \ref{s-example}.


\section{The continuity equation and the Wasserstein distance}\label{section 2}

In this section, we recall some properties of the continuity equation \eqref{eq:transport} and of the Wasserstein distance, which will be used all along this paper.

We denote by $\mc{P}_c(\mb{R}^d)$ the space of probability measures in $\mb{R}^d$ with compact support, and by $\mc{P}_c^{ac}(\mb{R}^d)$  the subset  of $\mc{P}_c(\mb{R}^d)$ of measures which are  absolutely continuous with respect to the Lebesgue measure. 
First of all, we give the definition of the push-forward of a measure and of the Wasserstein distance.
\begin{definition}
Denote by $\Gamma$  the set of the Borel maps $\gamma:\mb{R}^d\rightarrow\mb{R}^d$.
For a $\gamma\in\Gamma$, 
we define the push-forward $\gamma\#\mu$ of a measure $\mu$ of $\mb{R}^d$ as follows:
\vspace*{-1mm}\begin{equation*}
(\gamma\#\mu)(E):=\mu(\gamma^{-1}(E)),
\vspace*{-1mm}\end{equation*}
for every subset $E$ such that $\gamma^{-1}(E)$ is $\mu$-measurable.
\end{definition}

%

\begin{definition}
Let $p\in[1,\infty)$ and $\mu,\nu\in \mc{P}^{ac}_c(\mb{R}^d)$. Define 
\vspace*{-1mm}\begin{equation}\label{def:Wp}
W_p(\mu,\nu)=\inf\limits_{\gamma\in\Gamma}\left\{\left(\displaystyle\int_{\mb{R}^d}
|\gamma(x)-x|^pd\mu\right)^{1/p}:\gamma\#\mu=\nu\right\}.
\vspace*{-1mm}\end{equation}
\begin{prop}
$W_p$ is a distance on $\mc{P}^{ac}_c(\mb{R}^d)$, called the {\bf Wasserstein distance}. 
\end{prop}\label{prop Wp}
The Wasserstein distance can be extended to all pairs of measures $\mu,\nu$ compactly supported with the same mass $\mu(\mb{R}^d)=\nu(\mb{R}^d)\neq0$, by the formula
\vspace*{-3mm}$$W_p(\mu,\nu)=|\mu|^{1/p} W_p\left(\frac{\mu}{|\mu|},\frac{\nu}{|\nu|}\right).\vspace*{-1mm}$$
\end{definition}
For more details about the Wasserstein distance, in particular for its definition on the whole space of measures $ \mc{P}_c(\mb{R}^d)$, we refer to \cite[Chap. 7]{V03}.

 We now recall a standard result for the continuity equation: 
 \begin{theorem}[see \cite{V03}]
Let $T\in\mb{R}$,  $\mu^0\in \mc{P}^{ac}_c(\mb{R}^d)$ and  $w$ be a vector field uniformly bounded, Lipschitz in space and measurable in time.
Then the system 
\vspace*{-1mm}\begin{equation}\label{eq:transport sec 2}
	\left\{
	\begin{array}{ll}
\partial_t\mu +\nabla\cdot(w\mu)=0&\mbox{ in }\mb{R}^d\times\mb{R},\\\noalign{\smallskip}
\mu(\cdot,0)=\mu^0&\mbox{ in }\mb{R}^d
	\end{array}
	\right.
\vspace*{-1mm}\end{equation}
admits a unique solution\footnote{Here, $\mc{P}_c^{ac}(\mb{R}^d)$ is equipped with the weak topology, that coincides with the topology induced by the Wasserstein distance $W_p$, see \cite[Thm 7.12]{V03}.} $\mu$ in $\mc{C}^0([0,T];\mc{P}_c^{ac}(\mb{R}^d))$. Moreover, it holds $\mu(\cdot,t)=\Phi_t^{w}\#\mu^0$ for all $t\in \mb{R}$,
where the flow $\Phi_{t}^{w}(x^0)$ 
is the unique solution at time $t$ to
\vspace*{-1mm}\begin{equation}\label{eq charac}
\left\{\begin{array}{l}
\dot x(t) =w(x(t),t)\mbox{ for a.e. }t\geqslant 0,\\\noalign{\smallskip}
x(0)=x^0.
\end{array}\right.
\vspace*{-1mm}\end{equation}
 \end{theorem}

In the rest of the paper, the following properties of the Wasserstein distance will be helpful.
\begin{propi}[see \cite{PR13}]
Let $\mu,\nu\in\mc{P}_{c}^{ac}(\mb{R}^d)$. 
Let $w:\mb{R}^d\times\mb{R}\rightarrow\mb{R}^d$ be a vector field uniformly bounded, 
Lipschitz in space and measurable in time. For each $t\in\mb{R}$, it holds
\vspace*{-2mm}\begin{equation}\label{ine wasser 2}
W^p_p(\Phi_t^w\#\mu,\Phi_t^w\#\nu)
\leqslant e^{(p+1)L|t|} W^p_p(\mu,\nu),
\vspace*{-2mm}\end{equation}
where $L$ is the Lipschitz constant of $w$.
\end{propi}
\begin{propi}  Let $\mu,~\nu,~\rho,~\eta$ some positive measures satisfying 
$\mu(\mb{R}^d)=\nu(\mb{R}^d)$ and $\rho(\mb{R}^d)=\eta(\mb{R}^d)$. It then holds
\vspace*{-1mm}\begin{equation}\label{ine wasser}
W^p_p(\mu+\rho,\nu+\eta)
\leqslant W^p_p(\mu,\nu)+W^p_p(\rho,\eta).
\vspace*{-1mm}\end{equation}
\end{propi}
Using the properties of Wasserstein distance given in Section 1 of  \cite{PR13}, 
we can replace $W_p$ by $W_1$ in the definition of the approximate controllability.


%

\section{Proof of Theorem \ref{th1}}\label{sec:lip}

In this section, we prove  approximate controllability of System \eqref{eq:transport}. The proof is based on three approximation steps, corresponding to Proposition \ref{prop dim=d}, \ref{prop1}, and \ref{prop2}. The proof is then given at the end of the section.

In a first step, we suppose that the open connected control subset 
$\omega$ contains the support of both $\mu^0$, $\mu^1$.
%

\begin{prop}\label{prop dim=d}
Let $\mu^0,\mu^1\in\mc{P}_c^{ac}(\mb{R}^d)$ be such that $\supp(\mu^0)\subset\omega$ and 
$\supp(\mu^1)\subset\omega$.
Then, for  all  $T>0$, 
System \eqref{eq:transport} is approx. contr. at time $T$
with $\mathds{1}_{\omega}u$ in $\mc{U}$.
\end{prop}

\begin{proof}
We assume that  $d:=2$, $T:=1$ and $\omega:=(0,1)^2$, but
the reader will see that the proof can be clearly adapted to any space dimension.
Fix $n\in\mb{N}^*$. Define $a_0:=0$, $b_0:=0$ 
and the points $a_i,b_i$ for all $i\in\{1,...,n\}$ 
by induction as follows: suppose that for $i\in\{0,...,n-1\}$  the points $a_i$ and $b_i$
are given, then  $a_{i+1}$ and $b_{i+1}$ are the smallest values satisfying 
\vspace*{-1mm}\begin{equation*}
\begin{array}{c}\int_{(a_i,a_{i+1})\times\mb{R}}d\mu^0 =\frac{1}{n}
\mbox{ ~~~and~~~ }\int_{(b_i,b_{i+1})\times\mb{R}}d\mu^1 =\frac{1}{n}.
\vspace*{-1mm}\end{array}\end{equation*}
Again, for all $i\in\{0,...,n-1\}$, we define $a_{i,0}:=0$, $b_{i,0}:=0$
and supposing that for a  $j\in\{0,...,n-1\}$  
the points $a_{i,j}$ and $b_{i,j}$ are already defined, 
$a_{i,j+1}$ and $b_{i,j+1}$ are the smallest values such that
\vspace*{-1mm}\begin{equation*}
\begin{array}{c}
\int_{A_{ij}}d\mu^0 =\frac{1}{n^2}
\mbox{ ~~~and~~~ }
\int_{B_{ij}}d\mu^1 =\frac{1}{n^2},
\end{array}
\vspace*{-1mm}\end{equation*}
where $A_{ij}:=(a_i,a_{i+1})\times(a_{ij},a_{i(j+1)})$ 
and $B_{ij}:=(b_i,b_{i+1})\times(b_{ij},b_{i(j+1)})$.
Since $\mu^0$ and $\mu^1$ have a mass equal to $1$ and are supported in $(0,1)^2$, then 
$a_n,b_n\leqslant1$ and 
$a_{i,n},~b_{i,n}\leqslant1$ for all $i\in \{0,...,n-1\}$.
We give in Figure \ref{fig: mesh} an example of such decomposition.

\vspace*{-1mm}
\begin{figure}[h]
\begin{center}
\hspace*{-2mm}\begin{tikzpicture}[scale=2.1]
\draw[->] (0,0) -- (0,2.2);
\draw[->] (0,0) -- (3.2,0);
\draw[-] (0,2) -- (3,2);
\draw[-] (3,2) -- (3,0);
\path (0,2.3) node {$x_2$};
\path (3.3,0) node {$x_1$};
\path (0,-0.1) node {$a_0$};
\draw[-] (0.4,0) -- (0.40,2);
\path (0.4,-0.1) node {$a_1$};
\draw[-] (0,0.2) -- (0.4,0.2);
\path (-0.15,0.2) node {$a_{01}$};
\draw[-] (0,0.6) -- (0.4,0.6);
\path (-0.15,0.6) node {$a_{02}$};
\path (-0.1,1) node {$\vdots$};
\path (0.2,1) node {$\vdots$};
\draw[-] (0,1.5) -- (0.4,1.5);
\path (-0.3,1.5) node {$a_{0(n-2)}$};
\draw[-] (0,1.8) -- (0.4,1.8);
\path (-0.3,1.8) node {$a_{0(n-1)}$};
\path (-0.2,2) node {$a_{0n}$};
\draw[-] (0.7,0) -- (0.7,2);
\path (0.7,-0.1) node {$a_2$};
\draw[-] (0.4,0.1) -- (0.7,0.1);
\path (0.25,0.1) node {$a_{11}$};
\draw[-] (0.4,0.5) -- (0.7,0.5);
\path (0.25,0.5) node {$a_{12}$};
\path (0.55,1) node {$\vdots$};
\draw[-] (0.4,1.6) -- (0.7,1.6);
\draw[-] (0.4,1.7) -- (0.7,1.7);
\fill [opacity=0.5,pattern=north east lines] (0.4,0) -- (0.4,2) -- (0.7,2) -- (0.7,0);
\path (0.55,1.2) node {$\frac{1}{n}$};

\path (1,-0.1) node {$\cdots$};
\path (1,1) node {$\cdots$};
\draw[-] (1.3,0) -- (1.3,2);
\path (1.3,-0.1) node {$a_i$};

\draw[-] (1.3,0.2) -- (1.7,0.2);
\path (1.2,0.2) node {$a_{i1}$};

\path (1.5,0.5) node {$\vdots$};

\draw[-, thick] (1.3,0.9) -- (1.7,0.9);
\path (1.15,0.9) node {$a_{ij}$};
\draw[-, thick] (1.3,1.1) -- (1.7,1.1);
\path (1.05,1.1) node {$a_{i(j+1)}$};
\path (1.5,1.5) node {$\vdots$};
\draw[-, thick] (1.3,0.9) -- (1.3,1.1);
\draw[-, thick] (1.7,0.9) -- (1.7,1.1);

\fill [opacity=0.5,pattern=dots] (1.3,0.9) -- (1.3,1.1) -- (1.7,1.1) -- (1.7,0.9);
\path (1.5,1.0) node {$1/n^2$};

\draw[-] (1.3,1.8) -- (1.7,1.8);
\path (1.05,1.8) node {$a_{i(n-1)}$};

\draw[-] (1.7,0) -- (1.7,0.9);
\draw[-] (1.7,1.1) -- (1.7,2);
\path (1.7,-0.1) node {$a_{i+1}$};

\path (2,-0.1) node {$\cdots$};
\path (2,1) node {$\cdots$};

\draw[-] (2.3,0) -- (2.3,2);
\path (2.3,-0.1) node {$a_{n-2}$};

\draw[-] (2.3,0.1) -- (2.8,0.1);
\draw[-] (2.3,0.4) -- (2.8,0.4);
\path (2.5,1) node {$\vdots$};
\draw[-] (2.3,1.6) -- (2.8,1.6);
\draw[-] (2.3,1.7) -- (2.8,1.7);

\draw[-] (2.8,0) -- (2.8,2);
\path (2.8,-0.1) node {$a_{n-1}$};

\draw[-] (2.8,0.3) -- (3,0.3);
\draw[-] (2.8,0.5) -- (3,0.5);
\path (2.9,1) node {$\vdots$};
\draw[-] (2.8,1.5) -- (3,1.5);
\draw[-] (2.8,1.9) -- (3,1.9);
\path (3.05,-0.1) node {$a_n$};

\end{tikzpicture}
\caption{Example of a decomposition of $\mu^0$.}
\label{fig: mesh}
\end{center}
\end{figure}
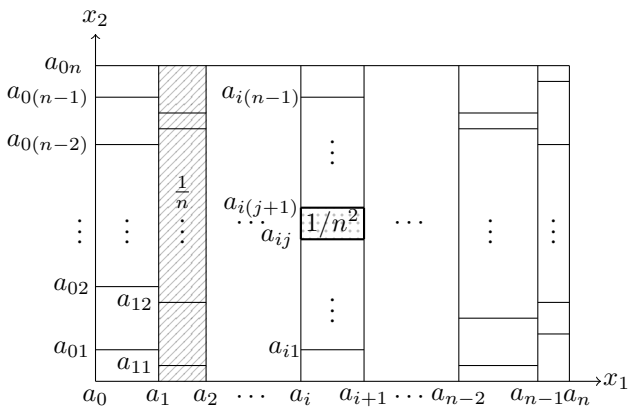
\vspace*{-1mm}

If one aims to define a vector field sending each $A_{ij}$ to $B_{ij}$, then some shear stress 
is naturally introduced to the interfaces of the cells. 
To overcome this problem, we first define sets $\widetilde{A}_{ij}\subset\subset A_{ij}$  
and  $\widetilde{B}_{ij}\subset\subset B_{ij}$
for all $i,j\in\{0,...,n-1\}$. We then send  the mass of $\mu^0$ from each $\widetilde{A}_{ij}$ to each $\widetilde{B}_{ij}$, while we do not control the mass contained in $A_{ij}\backslash\widetilde{A}_{ij}$.
More precisely, for all $i,j\in\{0,...,n-1\}$, we define, 
$a_i^-,~a_i^+,a_{ij}^-,~a_{ij}^+$
the smallest values such that
\vspace*{-1mm}\begin{equation*}
\hspace*{-2mm}\begin{array}{c}\int_{(a_i,a_{i}^-)\times(a_{ij},a_{i(j+1)})}d\mu^0 =
\int_{(a_i^+,a_{i+1})\times(a_{ij},a_{i(j+1)})}d\mu^0 =\frac{1}{n^3}
\end{array}\vspace*{-1mm}\end{equation*}
and
\vspace*{-1mm}\begin{equation*}
\begin{array}{rcl}
\int_{(a_i^-,a_{i}^+)\times(a_{ij},a_{ij}^-)}d\mu^0 &=&
\int_{(a_i^-,a_{i}^+)\times(a_{ij}^+,a_{i(j+1)})}d\mu^0\\ \noalign{\smallskip}
&=&\frac{1}{n}\times\left(\frac{1}{n^2}-\frac{2}{n^3}\right).
\end{array}\vspace*{-1mm}\end{equation*}
\noindent We similarly define $b_i^+,~b_i^-,~b_{ij}^+,~b_{ij}^-$.
We finally define
$$\widetilde{A}_{ij}:=[a_i^-,a_i^+)\times[a_{ij}^-,a_{ij}^+)
\mbox{ and }
\widetilde{B}_{ij}:=[b_i^-,b_i^+)\times[b_{ij}^-,b_{ij}^+).$$

The goal is to build a solution to System    \eqref{eq:transport}
such that the corresponding flow $\Phi_t^u$ satisfies  
\vspace*{-1mm}\begin{equation}\label{phi(Aij) in Bij}
\Phi_T^u(\widetilde{A}_{ij})=\widetilde{B}_{ij},
\vspace*{-1mm}\end{equation}
for all $i,j\in\{0,...,n-1\}$. 
We observe that we do not take into account the displacement of the mass contained in 
$A_{ij}\backslash \widetilde{A}_{ij}$.
We will show  that the corresponding term 
$W_1(\sum_{ij}\Phi_T^{v+u}\#\mu^0_{|A_{ij}\backslash\widetilde{A}_{ij}},\sum_{ij}\mu^1_{|B_{ij}\backslash\widetilde{B}_{ij}})$
tends to zero when $n$ goes to the infinity.
The rest of the proof is divided into two steps. 
In a first step, we build a flow and a velocity field such that its flow satisfies \eqref{phi(Aij) in Bij}.
In a second step, we compute the Wasserstein distance between $\mu^1$ and $\mu(T)$
showing that it converges to zero when $n$ goes to infinity.

{\bf Step 1:} 
We first build a flow satisfying \eqref{phi(Aij) in Bij}.
For all $i\in\{0,...,n-1\}$, we denote by $c^-_i$ and $c^+_i$
  the linear functions equal to $a_i^-$ and $a_i^+$ at time $t=0$ and equal to $b_i^-$
  and $b_i^+$ at time $t=T=1$, respectively
\textit{i.e.}
\vspace*{-1mm}\begin{equation*}
c^-_i(t)=(b_i^--a_i^-)t+a_i^-
\mbox{~and ~}
c^+_i(t)=(b_i^+-a_i^+)t+a_i^+.
\vspace*{-1mm}\end{equation*}
Similarly, for all $i,j\in\{0,...,n-1\}$, we denote by $c^-_{ij}$  and $c^+_{ij}$
the linear functions equal to $a_{ij}^-$ and $a_{ij}^+$ at time $t=0$ and equal to 
$b_{ij}^-$ and $b_{ij}^+$ at time $t=T=1$, respectively, \textit{i.e.}
\vspace*{-1mm}\begin{equation*}
c^-_{ij}(t)=(b_{ij}^--a_{ij}^-)t+a_{ij}^-
\mbox{~ and ~}
c^+_{ij}(t)=(b_{ij}^+-a_{ij}^+)t+a_{ij}^+.
\vspace*{-1mm}\end{equation*}

Consider the application being the following  linear combination of $c_i^-,~c_i^+$ and 
$c_{ij}^-,~c_{ij}^+$ in $\widetilde{A}_{ij}$, \textit{i.e.}
\vspace*{-1mm}\begin{equation}\label{expr char 2d}
x(x^0,t):=
\left(\begin{array}{c}
\dfrac{a_i^+-x^0_1}{a_i^+-a_i^-}c^-_i(t)
+\dfrac{x^0_1-a_i^-}{a_i^+-a_i^-}c^+_i(t)\\
\dfrac{a_{ij}^+-x^0_2}{a_{ij}^+-a_{ij}^-}c^-_{ij}(t)
+\dfrac{x^0_2-a_{ij}^-}{a_{ij}^+-a_{ij}^-}c^+_{ij}(t)
\end{array}\right),
\vspace*{-1mm}\end{equation}
when  $x^0\in \widetilde{A}_{ij}$. 
Let us prove that an extension of the application $(x^0,t)\mapsto \Phi_t^u(x^0):=x(x^0,t)$ is a flow associated to a velocity field $u$.
We remark that $t\mapsto x(x^0,t)$ is $\mc{C}^1$ and is solution to
\vspace*{-1mm} \begin{equation*}
\left\{\begin{array}{ll}
\frac{dx_{1}(x^0,t)}{dt}=\alpha_{i}(t)x_{1}(x^0,t)+\beta_i(t)&~\forall t\in[0,T],\\
\frac{dx_{2}(x^0,t)}{dt}=\alpha_{ij}(t)x_{2}(x^0,t)+\beta_{ij}(t)&~\forall t\in[0,T],
\end{array}\right.
\vspace*{-1mm}\end{equation*}
 where for all $t\in[0,1]$
\vspace*{-1mm}  \begin{equation*}
\left\{\begin{array}{l}
\alpha_i(t)=\frac{b_i^+-b_i^-+a_i^--a_i^+}{c^+_{i}(t)-c^-_{i}(t)}
,~
\beta_i(t)=\frac{a_i^+b_{i}-a_i^-b_i^+}{c_i^+(t)-c_i^-(t)},\\
\alpha_{ij}(t)=\frac{b_{ij}^+-b_{ij}^-+a_{ij}^--a_{ij}^+}{c^+_{ij}(t)-c^-_{ij}(t)}
,~
\beta_{ij}(t)=\frac{a_{ij}^+b_{ij}^--a_{ij}^-b_{ij}^+}{c_{ij}^+(t)-c_{ij}^-(t)}.
\end{array}\right.
\vspace*{-1mm}\end{equation*}
For all $t\in [0,1]$, consider the set 
$C_{ij}(t):=[c_i^-(t),c_i^+(t))\times [c_{ij}^-(t),c_{ij}^+(t)).$
We remark that $C_{ij}(0)=\widetilde{A}_{ij}$ and $C_{ij}(T)=\widetilde{B}_{ij}$.
On $C_{ij}:=\{(x,t): t\in [0,T], x\in C_{ij}(t)\}$, we then define the velocity field $u$ by
\vspace*{-1mm}  \begin{equation*}
u_1(x,t)=\alpha_{i}(t)x_1+\beta_i(t)\mbox{ and }
u_2(x,t)=\alpha_{ij}(t)x_2+\beta_{ij}(t),
\vspace*{-1mm}\end{equation*}
 for all $(x,t)\in C_{ij}$ ($x=(x_1,x_2)$). 
We extend $u$ by a $\mc{C}^{\infty}$ and uniformly bounded function outside $\cup_{ij}C_{ij}$,
then having $u\in\mc{U}$.
Then, System \eqref{eq:transport} admits a unique solution and the 
flow on $C_{ij}$ is given by the expression \eqref{expr char 2d}.

{\bf Step 2:} We now prove that the refinement of the grid provides convergence to the target $\mu_1$,  {\it i.e. }
\vspace*{-1mm}\begin{equation}\label{e-appcont}
W_1(\mu^1,\mu(T))\underset{n\rightarrow\infty}{\longrightarrow}0.
\vspace*{-1mm}\end{equation}
We remark that 
 \begin{equation*}
\begin{array}{c} \int_{\widetilde{B}_{ij}}d\mu(T)
 =\int_{\widetilde{B}_{ij}}d\mu^1
 =\frac{(n-2)^2}{n^4}.\end{array}
 \end{equation*}
Hence, by defining
$ R:=(0,1)^2~\backslash~ \bigcup\limits_{ij}\widetilde{B}_{ij},$
 we  also have
\vspace*{-1mm} \begin{equation*}
\begin{array}{c} \int_{R}d\mu(T) =\int_{R}d\mu^1
 =1-\frac{(n-2)^2}{n^2}.\end{array}
\vspace*{-1mm} \end{equation*}
It comes that 
\vspace*{-1mm} \begin{equation}\label{ine norm 1}
\begin{array}{c}
W_1(\mu^1,\mu(T))
\leqslant \sum\limits_{i,j=1}^nW_1(\mu^1\times\mathds{1}_{\widetilde{B}_{ij}},\mu(T)\times\mathds{1}_{\widetilde{B}_{ij}})\\
+W_1(\mu^1\times\mathds{1}_{R},\mu(T)\times\mathds{1}_{R}).
\end{array}
\vspace*{-1mm}\end{equation}
We estimate each term in the right-hand side. Since we deal with absolutely continuous measures, 
using Proposition \ref{prop Wp},
there exist measurable maps  
  $\gamma_{ij}:\mb{R}^2\rightarrow\mb{R}^2$, for all $i,j\in\{0,...,n-1\}$, 
  and   $\overline{\gamma} :\mb{R}^2\rightarrow\mb{R}^2$ such that
 \vspace*{-1mm}\begin{equation*}
 \gamma_{ij}\#(\mu^1\times\mathds{1}_{\widetilde{B}_{ij}})
 =\mu(T)\times\mathds{1}_{\widetilde{B}_{ij}}
 \vspace*{-1mm}\end{equation*}
and
  \vspace*{-1mm}\begin{equation*}
 \overline{\gamma}\#(\mu^1\times\mathds{1}_{R})
 =\mu(T)\times\mathds{1}_{R}.
\vspace*{-1mm} \end{equation*}
 In the first term, for each $i,j\in\{0,..., n-1\}$, observe that $\gamma_{ij}$ moves masses inside $B_{ij}$ only. Thus 
 \begin{equation}\label{ine norm 2}
\begin{array}{c}
W_1(\mu^1\times\mathds{1}_{\widetilde{B}_{ij}},\mu(T)\times\mathds{1}_{\widetilde{B}_{ij}})
= \int_{\widetilde{B}_{ij}}|x-\gamma_{ij}(x)|d\mu^1(x)\\
\leqslant (b_i^+-b_i^-+b_{ij}^+-b_{ij}^-)
\frac{(n-2)^2}{n^4}.
\end{array}
 \end{equation}
Concerning the second term in \eqref{ine norm 1}, observe that $\bar\gamma$ moves a small mass in the bounded set $\omega$. Thus it holds
 \begin{equation}\label{ine norm 3}
 \begin{array}{c}
 W_1(\mu^1\times\mathds{1}_{R},\mu(T)\times\mathds{1}_{R})
\leqslant\int_{R}|x-\overline{\gamma}(x)|d\mu^1(x)\\
\leqslant \sqrt{2}\left(1-\frac{(n-2)^2}{n^2}\right)
=4\sqrt{2}\frac{n-1}{n^2}.
\end{array}
 \end{equation}
We thus have \eqref{e-appcont} by combining \eqref{ine norm 1}, \eqref{ine norm 2} and \eqref{ine norm 3}. \end{proof}

In the rest of the section, 
 we remove the constraints $\supp(\mu^0)\subset \omega$ and 
$\supp(\mu^1)\subset \omega$, now imposing  Condition \ref{cond1}.
First of all, we give a consequence of Condition \ref{cond1}.
\begin{lemma}\label{lemma cond}
If Condition \ref{cond1}  is satisfied for $\mu^0,~\mu^1\in\mc{P}_c(\mb{R}^d)$, then the following Condition \ref{cond2} is satisfied too:
\end{lemma}
\begin{cond}\label{cond2}
There exist two real numbers $T_0^*$, $T_1^*>0$ and a non-empty open set $\omega_0\subset\subset\omega$ such that 
\begin{enumerate}
\item[(i)] For all $x^0\in \supp(\mu^0)$, 
there exists $t^0\in[0, T_0^*]$ such that $\Phi_{t^0}^v(x^0)\in \omega_0.$

\item[(ii)] For all $x^1\in \supp(\mu^1)$, 
there exists $t^1\in[0,T_1^*]$ such that $\Phi_{-t^1}^{v}(x^1)\in \omega_0$.
\end{enumerate}
\end{cond}

\begin{proof}
We use an compactness argument. Let $\mu^0\in\mc{P}_c(\mb{R}^d)$ and assume that Condition  \ref{cond1} holds.
Let  $x^0\in\supp(\mu^0)$.
Using Condition  \ref{cond1},
 there exists $t^0(x^0)>0$ such that $\Phi_{t^0(x^0)}^v(x^0)\in \omega.$
 Choose $r(x^0)>0$ such that $B_{r(x^0)}(\Phi_{t^0(x^0)}^v(x^0))\subset\subset \omega$, that exists since $\omega$ is open. By continuity of the application $x^1\mapsto \Phi_{t^0(x^0)}^v(x^1) $ (see  \cite[Th. 2.1.1]{BP07}), 
 there exists $\hat{r}(x^0)$ such that 
 $$x^1\in B_{\hat{r}(x^0)}(x^0)~~\Rightarrow~~\Phi_{t^0(x^0)}^v(x^1) \in B_{r(x^0)}(\Phi_{t^0(x^0)}^v(x^0)).$$
Since $\mu^0$ is compactly supported,  we can find a set $\{x^0_1,...,x^0_N\}\subset\supp(\mu^0)$ such that 
$$\supp(\mu^0)\subset \bigcup\limits_{i=1}^NB_{\hat{r}(x^0_i)}(x_i^0).$$
Thus the first item of Lemma \ref{lemma cond} is satisfied for 
\vspace*{-1mm}$$T_0^*:=\max\{ t^0(x^0_i)\}\}
~~~\mbox{ and }~~~
\omega_0:=\bigcup\limits_{i=1}^NB_{r(x^0_i)}(\Phi^v_{t^0(x^0_i)}(x^0_i)).\vspace*{-1mm}$$
The proof of the existence of $T_1^*$ is similar.
\end{proof}
We now prove that we can store nearly the whole mass of $\mu_0$ in $\omega$, under Condition \ref{cond2}.

\begin{prop}\label{prop1}
Let $\mu^0\in\mc{P}_c(\mb{R}^d)$ satisfying the first item of Condition \ref{cond2}.
Then there exists $\mathds{1}_{\omega}u\in\mc{U}$
such that 
\vspace*{-1mm}\begin{equation}\label{mu T0}
\supp(\mu(T_0^*))\subset\omega.
\vspace*{-1mm}\end{equation}
\end{prop}

\begin{proof}
Let $k\in\mb{N}^*$. 
We denote by $\alpha:=d(\omega,\omega_0)$, $\omega_1:=\{x^0\in \mb{R}^d:d(x^0,\omega_0)< \alpha/2\}$
and $S_k:=\{x^0\in \mb{R}^d:d(x^0,\omega_0)< \alpha/2k\}$.
 We define
$\theta_k$ a cutoff function
on $\omega$ of class $\mc{C}^{\infty}$ 
 satisfying
\vspace*{-1mm}\begin{equation}\label{def thetak}
\left\{\begin{array}{l}
0\leqslant \theta_k\leqslant 1,\\
\theta_k=1\mbox{ in }S_k^c,\\
\theta_k=0\mbox{ in }\omega_0.
\end{array}\right.
\vspace*{-1mm}\end{equation}
Define
\vspace*{-1mm}\begin{equation}\label{eq:control prop 35}
u_k:=(\theta_k-1) v.
\vspace*{-1mm}\end{equation} 
We remark that the support of $u_k$ is included in $\omega$.
Let $x^0\in\supp(\mu^0)$.
Define
\vspace*{-2mm}$$
t^*(x^0):=\inf\{t\in\mb{R}^+:\Phi_t^v(x^0)\in\overline\omega_0\}\leqslant T_0^*.\vspace*{-2mm}$$
Consider the flow $y:=\Phi_{t}^{v}(x^0)$ associated to $x^0$ without control, \textit{i.e.}
 the solution to 
\vspace*{-2mm}\begin{equation*}
\left\{\begin{array}{l}
\dot{y}(t)=v(y(t)),\\\noalign{\smallskip}
y(0)=x^0
\end{array}\right.
\vspace*{-2mm}\end{equation*}
and the flow  $z_k:=\Phi_{t}^{u_k+v}(x^0)$ associated to $x^0$ with the control $u_k$ 
given in \eqref{eq:control prop 35}, \textit{i.e.} the solution to
\vspace*{-2mm}\begin{equation}\label{eq:carac y}
\left\{\begin{array}{l}
\dot{z}_k(t)=(v+u_k)(z_k(t))=\theta_k(z_k(t))\times v(z_k(t)),\\\noalign{\smallskip}
z_k(0)=x^0.
\end{array}\right.
\vspace*{-2mm}\end{equation}
We now prove that the range of $z_k$ for $t\geq 0$ is included in the range of $y$ for $t\geq 0$.
Consider the solution $\gamma_k$ to the following system 
\vspace*{-2mm}\begin{equation}\label{eq:carac gamma}
\left\{\begin{array}{l}
\dot{\gamma_k}(t)=\theta_k(y(\gamma_k(t))),~t\geqslant 0,\\\noalign{\smallskip}
\gamma(0)=0.
\end{array}\right.
\vspace*{-2mm}\end{equation}
Since $\theta_k$ and $y$ are Lipschitz, then System \eqref{eq:carac gamma} admits a solution defined for all times.
We remark that $\xi_k:=y\circ \gamma_k$ is solution to System \eqref{eq:carac y}. 
Indeed for all $t\geqslant 0$
\vspace*{-1mm}\begin{equation*}
\left\{\begin{array}{l}
\dot\xi_k(t)=\dot\gamma_k(t)\times\dot y(\gamma_k(t))
=\theta_k(\xi_k(t))\times v(\xi_k(t)),\\\noalign{\smallskip}
\xi_k(0)=y(\gamma_k(0))=y(0).
\end{array}\right.
\vspace*{-1mm}\end{equation*}
By uniqueness of the solution to System \eqref{eq:carac y}, we obtain
\vspace*{-2mm}\begin{equation*}
y(\gamma_k(t))=z_k(t)\mbox{ for all }t\geqslant 0.
\vspace*{-2mm}\end{equation*}
Using the fact that $0\leqslant \theta\leqslant 1$ and the definition of $\gamma_k$, we have
 \vspace*{-1mm}$$\left\{\begin{array}{ll}
 \gamma_k\mbox{ increasing},&\\
 \gamma_k(t)\leqslant t&~\forall t\in[0,t^*(x^0)],\\
 \gamma_k(t)\leqslant t^*(x^0)&~\forall t\geqslant t^*(x^0).
 \end{array}\right. \vspace*{-1mm}$$
We deduce  that, for all $x^0\in\supp(\mu^0)$, 
\vspace*{-1mm}$$\{z_k(t):t\geqslant 0\}
 \subset\{y(s):s\in[0,t^*(x^0)]\}.\vspace*{-1mm}$$
We now prove that for all $k$ large enough, there exists $t\in(0,t^*(x^0))$ such that for all $s>t$, then $\Phi_s^{u_k+v}(x^0)\in \omega_1$.
 Consider  $B:=B_{\alpha/2}(\Phi_{t^*}^v(x^0))\subset \omega_1$.
By continuity, there exists $\beta>0$ such that $\Phi_t^v(x^0)\in B$ for all $t\in(t^*-\beta,t^*)$.
 For all $s\in[0,t^*-\beta]$, we can find $r(s)>0$ such that $B_{r(s)}(\phi_s^v(x^0))\subset \omega_0^c.$
By compactness of $\{\phi_s^v(x^0):s\in [0,t^*-\beta]\}$, there exists a finite subcover $\{B_{r(s_i)}(\phi_{s_i}^v(x^0))\}_{1\leqslant i\leqslant n}$ of $\{\phi_s^v(x^0):s\in [0,t^*-\beta]\}$.
We denote by $R:=\frac12 \min\{r(s_i)\}$.
Let $k$ be such that $\alpha/2k<R$.
Thus 
\vspace*{-1mm}\begin{equation*}
\left\{\begin{array}{ll}
\Phi_s^{v+u_k}(x^0)=\Phi_s^{v}(x^0),&\mbox{ for all }s\leqslant t^*-\beta,\\
\Phi_s^{v+u_k}(x^0)\in B\subset\omega_1,&\mbox{ for all }s> t^*-\beta.
\end{array}\right.
\vspace*{-1mm}\end{equation*}
There exists a ball $B_r(x^0)$, such that  $\Phi_s^{v+u_k}(x^1)\in\omega_1$ for all $x^1\in B_r(x^0)$ and $s>t^*-\beta$.
Thus, by compactness of $\supp(\mu^0)$, for $k$ large enough,  $\Phi_{T^*_0}^{v+u_k}(x^0)\in\omega_1$ for all $x^0\in \supp(\mu^0)$.
\end{proof}

The third step of the proof is to restrict a measure contained in $\omega$ to a measure contained in a hypercube $S\subset \omega$.

\begin{prop}\label{prop2}
Let $\mu^0\in\mc{P}_c(\mb{R}^d)$ satisfying  $\supp(\mu^0)\subset\omega.$
Define $S$ an open hypercube strictly included in $\omega$ and choose $\delta>0$. 
Then there exists $\mathds{1}_{\omega}u\in\mc{U}$
such that the corresponding solution to System \eqref{eq:transport} 
satisfies 
\vspace*{-2mm}$$\supp(\mu(\delta))\subset S.\vspace*{-1mm}$$
\end{prop}

\begin{proof}
From \cite[Lemma 1.1, Chap. 1]{FI96} and  \cite[Lemma 2.68, Chap. 2]{C09}, there exists 
 a function $\eta \in\mathcal{C}^2(\overline{\omega})$  satisfying 
\vspace*{-1mm}$$ \kappa_0\leqslant|\nabla\eta |\leqslant \kappa_1\mathrm{~in~}\omega\backslash S,~~~
  \eta>0\mathrm{~in~}\omega ~~~\mathrm{and}~~~
   \eta=0\mathrm{~on~}\partial\omega,\label{prop eta}\vspace*{-1mm}$$
with $\kappa_0,\kappa_1>0$. We extend $\eta$ by zero outside of $\omega$. 
$S\subset\subset\omega_k$. 
We denote by 
\vspace*{-1mm}\begin{equation*}
u_k:=k\nabla \eta .
\vspace*{-1mm}\end{equation*}
Let $x^0\in \supp(\mu^0)$. 
Consider the flow $z_k(t)=\Phi_t^{v+u_k}(x^0)$ associated to $x^0$,
\textit{i.e.} the solution to system 
\vspace*{-1mm}\begin{equation*}
\left\{\begin{array}{l}
\dot{z}_k(t)=v(z(t))+u_k(z_k(t)),~t\geqslant 0,\\\noalign{\smallskip}
z_k(0)=x^0.
\end{array}\right.
\vspace*{-1mm}\end{equation*}
The properties of $\eta$ imply that 
$n\cdot\nabla\eta<C<0$ on $\partial\omega$, 
where $n$ represents that exterior  normal vector to $\partial\omega$.
We deduce that, for $k$ large enough, $n\cdot(v+k\nabla \eta)<0$ on $\partial\omega$.
Thus $z_k(t)\in\omega$ for all $t\geqslant 0$.

We now prove that there exists $K\in\mb{N}^*$ and $T\in(0,\delta)$ such that for all $k>K$ and $t\in[T,\delta]$, 
$z_k(t)\in S$ for all $x^0\in\supp(\mu^0)$. By contradiction, assume that 
there exists three sequences $\{k_n\}_{n\in\mb{N}^*}\subset \mb{N}^*$, $\{t_n\}_{n\in\mb{N}^*}\subset(0,\delta)$ and $\{x^0_n\}_{n\in\mb{N}^*}\in\supp(\mu^0)$ satisfying 
$k_n \rightarrow \infty$, $t_n\rightarrow \delta$ and
\vspace*{-2mm}\begin{equation}\label{eq:Phi not in omega}
z_{k_n}(x^0_n,t_n)\in S^c.
\vspace*{-2mm}\end{equation}
Consider the function $f_n$ defined  for all $t\in[0,\delta]$ by
\vspace*{-2mm}$$f_n(t):=\eta(z_{k_n}(t)).\vspace*{-2mm}$$
Its time derivative is given by
\vspace*{-2mm}$$\dot f_n(t)
=k_n|\nabla \eta(z_{k_n}(t))|^2+v(z_{k_n}(t))\cdot\nabla \eta(z_{k_n}(t)).\vspace*{-2mm}$$
Then, using \eqref{eq:Phi not in omega} and 
the properties  of $\eta$, it holds
\vspace*{-2mm}$$f_n(t_n)\geqslant (k_n\kappa_0^2-\|v\|_{L^{\infty}(\omega)}\kappa_1)t_n,\vspace*{-2mm}$$
which is in contradiction for $n$ large enough with
\vspace*{-2mm}$$f_{k_n}(t_n)\leqslant\|\eta\|_{\infty}.\vspace*{-2mm}$$
Thus we deduce that, for a $K\in\mb{N}^*$ and a $T\in[0,\delta]$,
 $\Phi_t^{v+u_k}(x^0)\in S$ for all  $x^0\in \supp(\mu^0)$, $t\in(T,\delta)$ and $k>K$.

%
%

\end{proof}

We now have all the tools to prove Theorem \ref{th1}. The idea is the following: we first send $\mu_0$ inside $\omega$ with a control $u_1$, then from $\omega$ to an hypercube $S$ with a control $u_2$. On the other side, we send $\mu_1$ inside $\omega$ backward in time with a control $u_5$, then from $\omega$ to $S$ with a control $u_4$. When both the source and the target are in $S$, we send one to the other with a control $u_3$.

{\it Proof of Theorem \ref{th1}:} Consider $\mu^0,\mu^1$ satisfying Condition  \ref{cond1}. Then, by Lemma \ref{lemma cond}, there exist $T_0^*,T_1^*$ for which  $\mu^0,\mu^1$ satisfy Condition \ref{cond2}. Define $T:=T_0^*+T_1^*+\delta$ with 
$\delta>0$. 

Choose $\varepsilon>0$ and denote by $T_1:=T_0^*$,  $ T_2:=T_0^*+\delta/3$, $T_5:=T_1^*$ and  $ T_4:=T_1^*+\delta/3$.
Using Propositions \ref{prop1} and \ref{prop2}, there exists some controls $u^1,~u^2,~u^4,~u^5\in\mc{U}$  and a square $S\subset \omega$ such that the solutions to
\vspace*{-1mm}\begin{equation}\label{eq:approx step1}
	\left\{
	\begin{array}{ll}
\partial_t\rho_0 +\nabla\cdot((v+\mathds{1}_{\omega}u^1)\rho_0)=0
&\mbox{ in }\mb{R}^d\times[0,T_1],\\\noalign{\smallskip}
\partial_t\rho_0 +\nabla\cdot((v+\mathds{1}_{\omega}u^2)\rho_0)=0
&\mbox{ in }\mb{R}^d\times[T_1,T_2],\\\noalign{\smallskip}
\rho_0(0)=\mu^0&\mbox{ in }\mb{R}^d
	\end{array}
	\right.
\vspace*{-2mm}\end{equation}
and
\vspace*{-1mm}\begin{equation}\label{eq:approx step2}
	\left\{
	\begin{array}{ll}
\partial_t\rho_1 +\nabla\cdot((v+\mathds{1}_{\omega}u^5)\rho_1)=0
&\mbox{ in }\mb{R}^d\times[-T_5,0],\\\noalign{\smallskip}
\partial_t\rho_1 +\nabla\cdot((v+\mathds{1}_{\omega}u^4)\rho_1)=0
&\mbox{ in }\mb{R}^d\times[-T_4,-T_5],\\\noalign{\smallskip}
\rho_1(0)=\mu^1&\mbox{ in }\mb{R}^d,
	\end{array}
	\right.
\vspace*{-1mm}\end{equation}
satisfy $\supp(u_i)\subset\omega$,
$\rho_0(T_2)(S)>1-\varepsilon$ and $\rho_1(-T_4)(S)>1-\varepsilon.$

We now apply Proposition \ref{prop dim=d} to approximately steer $\rho_0(T_2)$ to $\rho_1(-T_4)$ inside $S$: this gives a control $u_3$ 
on the time interval $[0,\frac{\delta}3]$. Thus, concatenating $u_1,u_2,u_3,u_4,u_5$ on the time interval $[0,T]$, we approximately steer $\mu_0$ to $\mu_1$. \hfill\QED

\section{Proof of Theorem \ref{th2}}\label{s-proof2}

In this section, we prove Theorem \ref{th2} about minimal time. To achieve controllability in this setting, one needs to store the mass coming from $\mu^0$ in $\omega$ and to send it out with a rate adapted to approximate $\mu^1$.

Let $T_0$ be the infimum satisfying Condition \eqref{cond temps opt}, and fix $s>0$.
We now prove that System \eqref{eq:transport} is approximately controllable at time $T:=T_0+s$.
Consider $N\in\mb{N}^*$,    $\tau:=T_0/N$,  $\delta<\tau$, $\xi:=\tau-\delta$ and $\tau_i:=i\times \tau$. 
Define
\vspace*{-1mm}\begin{equation*}
\left\{\begin{array}{l}
A_i:=\{x^0\in\supp(\mu^0):t^0(x^0)\in[0,\tau_{i})\},\\
B_i:=\{x^1\in\supp(\mu^1):T_0-t^1(x^1)\in[\tau_i,\tau_{i+1})\},
\end{array}\right.
\vspace*{-1mm}\end{equation*}
where 
$\left\{\begin{array}{l}
t^0(x^0):=\inf\{t\in\mb{R}^+:\Phi_t^v(x^0)\in \omega\},\\ 
t^1(x^1):=\inf\{t\in\mb{R}^+:\Phi_{-t}^v(x^1)\in \omega\}.
\end{array}\right.$\\
We remark that $\mu^0\times\mathds{1}_{A_i}$ represents the mass of $\mu^0$ 
which has entered $\omega$ at time $\tau_i$ 
and $\mu^1\times\mathds{1}_{B_i}$ the mass of $\mu^1$  
which need to exit $\omega$ in the time interval $(\tau_i,\tau_{i+1})$.
Then, by hypothesis of the Theorem, there exists $K$ such that
\vspace*{-1mm}\begin{equation*}
(\Phi^{v+u_K}_{\tau_i}\#(\mu^0\times\mathds{1}_{A^0_i}))(\omega)
\geqslant 1-(\Phi^{v+u_K}_{\tau_i-T}\#(\mu^1\times\mathds{1}_{A^1_i}))(\omega)-\varepsilon.
\vspace*{-1mm}\end{equation*}
The function $u_K$ can be then used to store the mass of $\mu^0$ in $\omega$. The meaning of the previous equation is that the stored mass is sufficient to fill the required mass for $\mu^1$.

We now define the control achieving approximate controllability at time $T$ as follows:
First of all, using the same strategy as in the Proof of Theorem \ref{th1}, 
we can send a part of $\phi_{s-\xi}^{v+u_K}\#( \mu^0\times\mathds{1}_{A_0})$
approximately to $\phi_{-T^*}^{v+u_K}\#( \mu^1\times\mathds{1}_{B_0})$
during the time interval $(s-\xi,s)$.
More precisely, we replace $T_0^*$ and $T_1^*$ by $s-\xi$ and $\xi$ in the proof of Theorem \ref{th1}.
Thus, we send the mass of $\mu^0$ contained in $A_0$
near to the mass of $\mu^1$ contained in $B_0$.
We repeat this process on each time interval $(\tau_i,\tau_{i+1})$ for  $A_i$ to $B_i$.
Thus, the mass of $\mu^0$ is globally sent close to the mass of $\mu^1$ in time $T$.

\section{Example of minimal time problem}
\label{s-example}

In this section, we give explicit controls realizing the approximate minimal time in one simple example. The interest of such example is to show that the minimal time can be realized by non-Lipschitz controls, that are unfeasible.

We study an example on the real line. We consider a constant initial data $\mu_0=\mathds{1}_{[0,1]}$ and a constant uncontrolled vector field $v=1$. The control set is $\omega=[2,3]$. Our first target is the measure $\mu_1=\frac12 \mathds{1}_{[4,6]}$. We now consider the following control strategy:
\vspace*{-2mm}\begin{equation}
u(t,x)=\begin{cases}
0&\hspace*{-3mm}t\in[0\frac43)\cap[\frac73,\frac{13}3],\\
\psi(2+(t-\frac43),\frac73+2(t-\frac43))&t\in[\frac43,\frac53),\\
\psi(2+(t-\frac53),\frac73+2(t-\frac53))&t\in[\frac53,2),\\
\psi(2+(t-2),\frac73+2(t-2))&t\in[2,\frac73),
\end{cases}\label{e-u1}
\vspace*{-2mm}\end{equation}
where $\psi(a,b)$ is defined as follows:
\vspace*{-2mm}\begin{equation}
\psi(a,b)(x)=\begin{cases}
\frac{x-a}{b-a}& x\in[a,b],\\
0 &x\not\in[a,b].
\end{cases}
\vspace*{-2mm}\end{equation}
The choice of $\psi(a,b)$ given above has the following meaning: the vector field $\psi(a,b)$ is linearly increasing  on the interval, thus an initial measure with constant density $k\mathds{1}_{[\alpha_0,\beta_0]}$ with $a\leq \alpha_0\leq \beta_0\leq b$ will be transformed to a measure with constant density, supported in $[\alpha(t),\beta(t)]$, where $\alpha(t)$ is the unique solution of the ODE 
\vspace*{-2mm}$$
\begin{cases}
\dot x=v+u,\\
x(0)=\alpha_0,
\end{cases}\label{e-odemia}
\vspace*{-2mm}$$
and similarly for $\beta(t)$. As a consequence, we can easily describe the solution $\mu(t)$ of \eqref{eq:transport} with control \eqref{e-u1} and initial data $\mu_0$. For simplicity, we only describe the measure evolution and the vector field on the time interval $[1,\frac43]$ in Figure \ref{f-rarefaction}. One can observe that the linearly increasing time-varying control allows to rarefy the mass.

Two remarks are crucial:
\begin{itemize}
\item The vector field $v+\psi(a,b)$ is not Lipschitz, since it is discontinuous. Thus, one needs to regularize such vector field with a Lipschitz mollificator. As a consequence, the final state does not coincide with $\mu_1$, but it can be chosen arbitrarily close to it;
\item The strategy presented here cuts the measure in three slices of mass $\frac13$, and rarefying each of them separately. Its total time is $4+\frac13$. One can apply the same strategy with a larger number $n$ of slices, and rarefying the mass in $[2,2+\frac1n]$ by choosing the control $\psi(2+t,2+\frac1n +2t)$. With this method, one can reduce the total time to $4+\frac1n$, then being approximately close to the minimal time $T_0=4$ given by Theorem \ref{th2}.
\end{itemize}

\vspace*{-5mm}
\begin{figure}[h]\begin{center}
   \hspace*{-5mm} \includegraphics[width=10cm]{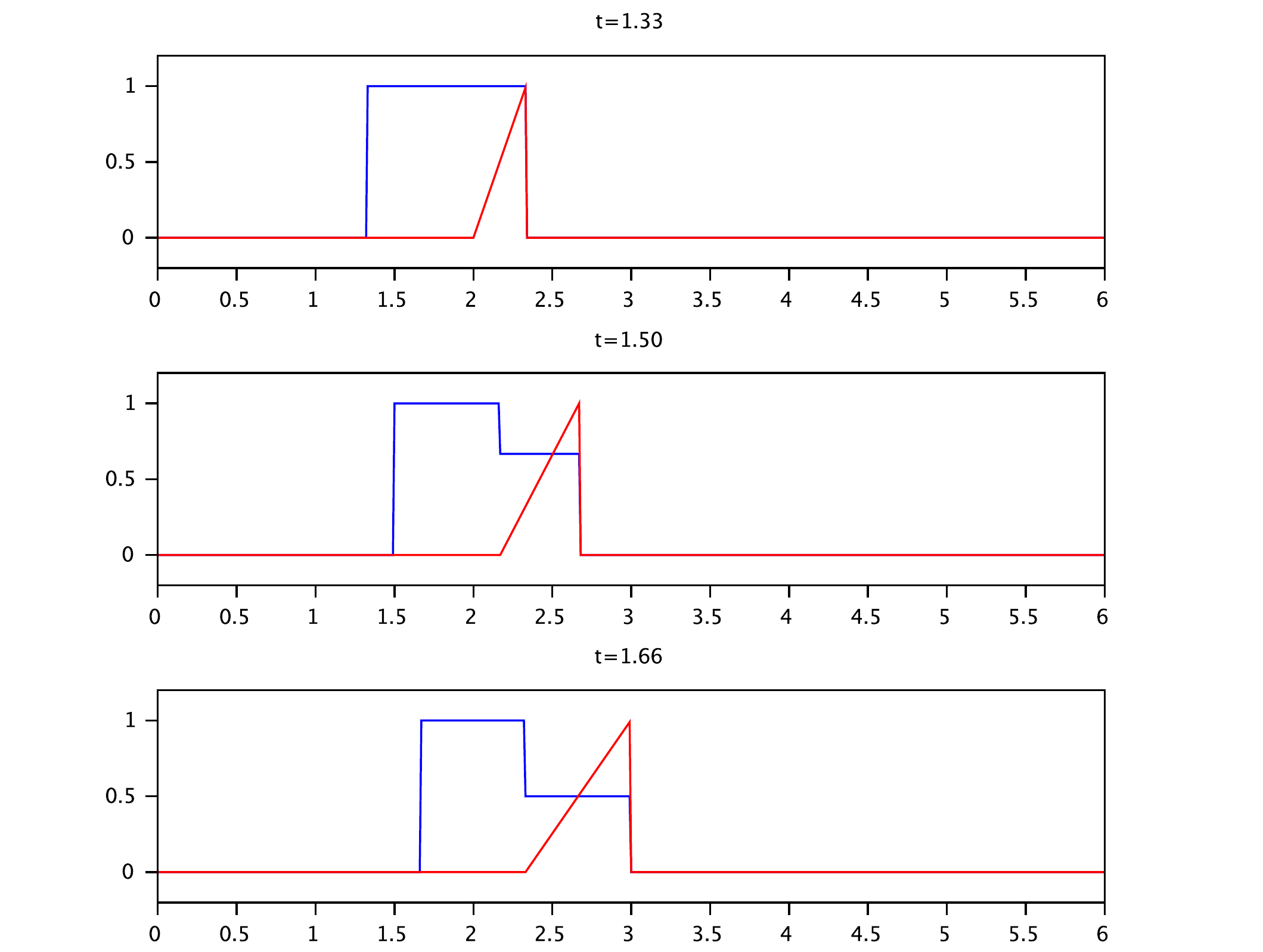}
\caption{Blue: density of the measure. Red: control vector field.}
\vspace*{-1mm}\label{f-rarefaction}
\end{center}\end{figure}

\section{Conclusion}

In this article, we studied the control of a transport equation, where the control is a Lipschitz vector field in a fixed set $\omega$. We proved that approximate controllability can be achieved under reasonable geometric conditions for the uncontrolled systems. We also proved a result of minimal time control from one configuration to another.
Future research directions include the study of more general transport equations, namely when the uncontrolled dynamics presents interaction terms, such as in models for crowds and opinion dynamics.

\bibliographystyle{plain}
\bibliography{biblio.bib}

\end{document}